\documentclass[12pt,english,reqno]{amsart}  

\usepackage[T1]{fontenc} 
\usepackage{geometry} 
\usepackage{amssymb,amsmath,amsthm} 

\usepackage{pdfpages}
\usepackage{graphicx}
\usepackage{tikz} 
\usetikzlibrary{arrows,automata} 

\usepackage{hyperref}
\hypersetup{
    colorlinks=true,
    linkcolor=blue,
    filecolor=magenta,      
    urlcolor=cyan,
}

\usepackage{float} 

\newtheorem{theorem}{Theorem}[section]

\newtheorem{lemma}[theorem]{Lemma}

\newtheorem{definition}[theorem]{Definition}
\newtheorem{remark}[theorem]{Remark}

\title{Navigating the Space of Compact CMC Hypersurfaces in Spheres, Part I}
\author{Oscar Perdomo}
\address{Central Connecticut State University}
\email{perdomoosm@ccsu.edu}
\date{\today}

\begin{document}

\maketitle

\begin{abstract}
In this paper, we describe a family of embedded hypersurfaces with constant mean curvature (CMC) in the $(n+1)$-dimensional unit sphere. In the process, we provide evidence for new CMC embedded examples. In particular, for some examples with $H=0$, we verify Yau's conjecture stating that among the embedded, non-totally umbilical minimal hypersurfaces in spheres, the Clifford hypersurfaces have the least area.
\end{abstract}

\section{Introduction}

For any positive integer \( m \), we denote by \( S^m \) the unit \( m \)-dimensional sphere in the Euclidean space \( \mathbb{R}^{m+1} \). This paper builds upon results from the following works:
\begin{itemize}
    \item \cite{CS}, where the authors establish the existence of minimal immersions \( S^l \times S^l \times S^1 \to S^{2k+1} \).
    \item \cite{HW}, where the authors show the existence of CMC immersions of \( S^l \times S^l \times S^1 \longrightarrow S^{2k+1} \).
    \item \cite{P}, in which minimal immersions \( S^k \times S^l \times S^1 \to S^{k+l+1} \) are studied and a method is developed to compute the spectra of both the stability and Laplace operators.
\end{itemize}

In this paper, we continue the study initiated in \cite{P} of CMC hypersurfaces in spheres. We present evidence that for every pair \((k,l)\) there exists a continuous one-parameter family of embedded hypersurfaces with constant mean curvature \(H\) in \(S^{n+1}\), where \(n=k+l+1\). In this family, the parameter \(H\) takes values in an interval of the form \([H_{n\ell}^{\min}, \infty)\) with \(H_{n\ell}^{\min} < 0\), and every negative value (except \(H_{n\ell}^{\min}\)) is attained twice.

More precisely, the family begins with CMC hypersurfaces whose mean curvature \(H\) is negative and close to zero. These hypersurfaces resemble Kapouleas' construction: they feature two “catenoidal necks.” In this configuration, one neck is topologically of the form \(S^k \times (-1,1)\) and the other is of the form \(S^l \times (-1,1)\). As the parameter varies, the family transitions into embeddings of \(S^k \times S^l \times S^1\) in which none of the three spherical factors have their radii  ``near'' zero. In this regime, \(H\) decreases from negative values close to zero to the minimum possible value in the family, \(H_{n\ell}^{\min}<0\); then \(H\) continues to increase, passing through \(H=0\) (the only minimal embedding in the family), and finally, as \(H\) tends to infinity, the hypersurfaces resemble the Cartesian product of an \((n-1)\)-dimensional Clifford hypersurface with a closed curve of very small length.

 There are many interesting questions regarding the volume of compact (without boundary) minimal hypersurfaces in spheres.  It is known that the totally geodesic \(S^{m-1}\) in \(S^m\) has the smallest volume among all such hypersurfaces. In other words, if we denote by \(\sigma_m\) the volume of the unit \(m\)-dimensional sphere, then every minimal non-totally geodesic compact hypersurface \(M \subset S^{n+1}\) (without boundary) has volume greater than \(\sigma_n\).

It has been conjectured by Yau (page 288 of \cite{Y}) that among all compact minimal hypersurfaces in \(S^{n+1}\) that are not totally geodesic, the second lowest volume is achieved by a minimal Clifford hypersurface. As he stated in \cite{Y}:

\begin{minipage}{0.9\textwidth}
\centering
``Does the volume of the minimal hypersurfaces given by the product of spheres give the lowest value of the volume among all non-totally geodesic minimal hypersurfaces? B. Solomon thinks this is the case.''
\end{minipage}

Yau's conjecture on the volume of compact minimal hypersurfaces has so far been verified only in the case of surfaces in \(S^3\) (see \cite{MN}). The family that we are studying in this paper includes, for low dimensions, one minimal compact embedded example in \(S^4\), one in \(S^5\), two in \(S^6\), two in \(S^7\), three in \(S^8\), three in \(S^9\), four in \(S^{10}\), four in \(S^{11}\), five in \(S^{12}\) and  five in \(S^{13}\). In Section \ref{area}, we will compute the volume of these examples and verify that Yau's conjecture holds true for them.

Since this paper focuses on families of embedded CMC hypersurfaces in spheres, we begin by describing the three easiest-to-describe families of such hypersurfaces.

\subsection{Totally Umbilical Hypersurfaces} 
For any \( k \), we define the immersion 
\[
\phi: S^k \longrightarrow S^{k+1}
\]
by
\begin{eqnarray}\label{imtu}
\phi(y) = \left(r\, y, \sqrt{1 - r^2} \right),
\end{eqnarray}
where \( r \) is a constant in \( (0,1) \) and \( y \) represents a point in \( S^k \). With respect to the Gauss map \( \nu = (-\sqrt{1 - r^2}\, y, r) \), the mean curvature of the hypersurface is given by 
\[
H = \frac{\sqrt{1 - r^2}}{r}.
\]

\begin{figure}[h]
\centerline{\includegraphics[scale=0.52]{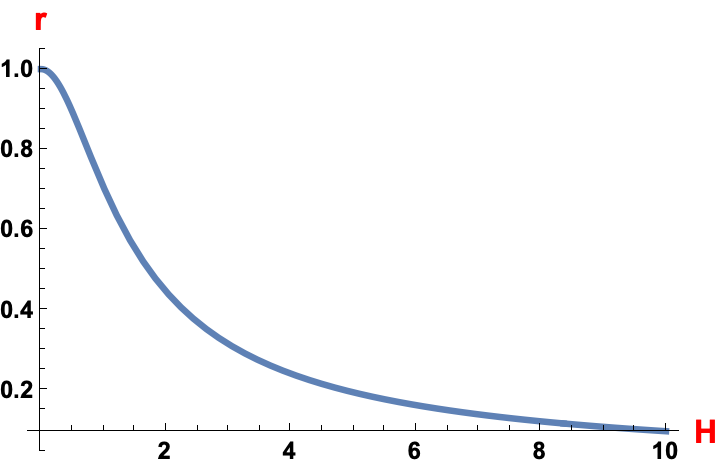}}
\caption{Each point on the curve \( r = \frac{1}{\sqrt{1+H^2}} \) represents a CMC hypersurface. Specifically, the point \( \left(H, \frac{1}{\sqrt{1+H^2}}\right) \) corresponds to the hypersurface given by the immersion (\ref{imtu}) with \( r = \frac{1}{\sqrt{1+H^2}} \).}
\label{pcHeq0neq3}
\end{figure}

\subsection{Clifford hypersurfaces} 
For any pair of positive integers \( k \) and \( l \), we define the immersion 
\[
\phi: S^k \times S^l \longrightarrow S^{k+l+1}
\]
by
\[
\phi(y,z) = \left( r\, y, \sqrt{1 - r^2}\, z \right),
\]
where \( r \) is a constant in \( (0,1) \), and \( y \) and \( z \) represent points in \( S^k \) and \( S^l \), respectively.

With respect to the Gauss map \( \nu = (-\sqrt{1 - r^2}\, y, r\, z) \), the mean curvature \( H \) satisfies
\[
(k+l) H = k \frac{\sqrt{1 - r^2}}{r} - l \frac{r}{\sqrt{1 - r^2}}.
\]

This family of hypersurfaces realizes every possible value of \( H \). As \( H \to -\infty \), the projection of the hypersurface onto the last \( l+1 \) coordinates forms an \( l \)-sphere whose radius approaches zero. Conversely, as \( H \to \infty \), the projection onto the first \( k+1 \) coordinates forms a \( k \)-sphere with a radius approaching zero.

Notably, to account for all possible hypersurfaces in this family, we must consider both positive and negative values of \( H \). Furthermore, in general, if \( k \neq l \), the hypersurface associated with \( H \) is not isometric to the one associated with \( -H \).

\subsection{Rotational hypersurfaces} 
For any positive integer \( k \), we define the immersion 
\[
\phi: S^k \times \mathbb{R} \longrightarrow S^{k+2}
\]
by
\begin{eqnarray}\label{rotEx}
\phi(y,t) = \left( \sqrt{1 - f_1(t)^2 - f_2(t)^2}\, y,\, f_1(t),\, f_2(t) \right),
\end{eqnarray}
where \( \alpha(t) = (f_1(t), f_2(t)) \) is a closed curve contained within the interior of the unit disk, called the \emph{profile curve}. Here, \( y \) represents a point in \( S^k \).

The minimal examples (\( H = 0 \)) in this family were studied by Otsuki \cite{O1, O2}, while the non-minimal examples were considered in \cite{PRe, WCL, DD, BL}. The family of embedded CMC rotational hypersurfaces can be described as follows (see \cite{PRe}):  

For any integers \( k \geq 1 \) and \( m \geq 2 \), there exist curves 
\[
\alpha_{H,k,m}(t) = (f_1(t), f_2(t)) \quad\text{with } H \text{ between } \cot\left(\frac{\pi}{m}\right) \quad\text{and}\quad \frac{\sqrt{k}(m^2 - 2)}{(k+1)\sqrt{m^2 - 1}},
\]
such that the immersion \eqref{rotEx} defines an embedded hypersurface in \( S^{k+2} \) with constant mean curvature \( H \).

Notably, when \( k = 1 \), no surface in this family exists with \( H = \frac{1}{\sqrt{3}} \). Another important observation is that for any \( m \geq 2 \), the profile curve \( \alpha_{H,k,m}(t) = (f_1(t), f_2(t)) \) closely resembles a regular \( m \)-gon with vertices on the unit circle when \( H \) is near \( \cot\left(\frac{\pi}{m}\right) \). Conversely, when \( H \) is close to the other bound,
\[
H = \frac{\sqrt{k}(m^2 - 2)}{(k+1)\sqrt{m^2 - 1}},
\]
the profile curve \( \alpha_{H,k,m}(t) \) becomes nearly circular.

\section{The new family of CMC in the sphere }

In this section, we introduce the main family of hypersurfaces that will be the focus of this paper. These hypersurfaces generalize the classical rotational hypersurfaces by incorporating an additional parameter, leading to new examples of embedded CMC hypersurfaces in spheres.

Let us consider an immersion of the form $\phi:S^k\times S^l\times \mathbb{R}\longrightarrow S^{n+1}$ with 
\begin{eqnarray}\label{param}
\phi(y,z,t)=(\sqrt{1-f_1(t)^2-f_2(t)^2}\, y,f_2(t)z,f_1(t))
\end{eqnarray}
where $f_1$ and $f_2$ are $2T$-periodic functions that satisfy $f_2(t)>0$ and $f_1(t)^2 + f_2(t)^2 < 1$. We are calling the period $2T$ instead of $T$ because we will be working more with half of the period of the profile curve. Here, $k+l+1=n$, with $y$ representing a point in $S^{k}$ and $z$ representing a point in $S^{l}$. 

In \cite{P}, the author shows that the mean curvature $H$ of the immersion $\phi$ satisfies
\begin{eqnarray}\label{ExpH}
nH=\frac{ng+\kappa_1+l\kappa_2}{h}-\frac{1}{h^3}(\kappa_1(ff^\prime)^2)
\end{eqnarray}
with
\begin{eqnarray*}
f=\sqrt{1-f_1^2-f_2^2}\, ,\quad g=f_2f_1^\prime-f_1f_2^\prime , \quad h=\sqrt{1-g^2}\\
\kappa_1= f_2^{\prime\prime}(t)f_1^\prime(t)-f_1^{\prime\prime}(t)f_2^\prime(t)
\quad\hbox{and}\quad \kappa_2=-\frac{f_1^\prime(t)}{f_2(t)}.
\end{eqnarray*}

We will call the curve $\alpha(t)=(f_1(t),f_2(t))$ the profile curve of the immersion. In this case, the mean curvature is taken with respect to the Gauss map
\begin{eqnarray}
\xi=h^{-1}\left( -gf y,(-gf_2+f_1^\prime)\, z,-gf_1-f_2^\prime\right).
\end{eqnarray}

In \cite{CS}, Carlotto and Schulz proved the existence of embedded minimal examples of the form (\ref{param}) in the case when \( k = l \). Similarly, in the same setting \( k = l \), Huang and Wei established the existence of embedded examples with constant mean curvature \cite{HW}. To demonstrate the existence of embedded examples, both papers considered as the profile curve the curve in 2-dimensional sphere,

\[
\beta = (\sqrt{1 - f_1^2 - f_2^2}, f_2, f_1) = \left(\sin(u_2)\cos(u_1), \sin(u_2)\sin(u_1),\cos(u_2)\right),
\]

parametrized by arc-length. In this case, the corresponding ODE system reduces to

\begin{eqnarray}\label{thesystemWei}
\begin{cases}
u_1^\prime &= \frac{\sin(u_3)}{\sin(u_2)}, \\
u_2^\prime &= \cos(u_3), \\
u_3^\prime &= 2k\frac{\cot(2 u_1)}{\sin(u_2)}-(2k+1)\cot(u_2)\sin(u_3)-(k+2) H.
\end{cases}
\end{eqnarray}

For the case \( k = l \), we extend the solution provided in \cite{CS} and \cite{P} to a one-parameter family that starts near \( H = 0 \) and includes both positive and negative values of \( H \). In contrast, the subfamily extending from \( H = 0 \) for positive values \( H > 0 \) was previously considered in \cite{HW}, but the case \( H < 0 \) was not addressed in their work. Our extension provides a full description of the family by incorporating both positive and negative values of \( H \).

\subsection{The ODE} 

In this subsection, we set up the system of differential equations that generates CMC hypersurfaces of the form \eqref{param}.

We can assume, without loss of generality, that the curve \( \alpha(t) = (f_1(t), f_2(t)) \) is parametrized by arc-length. That is, we assume that for some function \( \theta(t) \), the following holds:
\[
f_1^\prime = \cos\theta \quad \hbox{and} \quad f_2^\prime = \sin\theta.
\]
Substituting the identities
\[
f_1^{\prime\prime} = -\theta^\prime\sin\theta, \quad f_2^{\prime\prime} = \theta^\prime\cos\theta
\]
into Equation \eqref{ExpH}, we obtain that \( \theta^\prime = K \), where \( K \) is given in terms of \( \theta \), \( f_1 \), and \( f_2 \) by
\begin{eqnarray}\label{dth}
K=\frac{h^2}{f_2f^2}\, \left( n f_1 f_2 \sin \theta+H h n f_2-n f_2^2 \cos \theta+l \cos \theta \right). 
\end{eqnarray}

Here,
\[
g = f_2\cos\theta - f_1\sin\theta, \quad h = \sqrt{1 - g^2}.
\]

We now state the following theorem.

\begin{theorem} \label{theODE}
Let \( k \) and \( l \) be positive integers such that \( n = k + l + 1 \). If \( f_1 \) and \( f_2 \) are \( T \)-periodic functions satisfying \( f_2 > 0 \) and \( f_1^2 + f_2^2 < 1 \), and they satisfy the system
\begin{eqnarray}\label{thesystem}
\begin{cases} 
f_1^\prime &= \cos(\theta),\\
f_2^\prime &= \sin(\theta),\\
\theta^\prime &= \frac{h^2}{f_2f^2}\, \left( n f_1 f_2 \sin \theta+n H f_2 h-n f_2^2 \cos \theta+l \cos \theta \right),
\end{cases}
\end{eqnarray}
where
\[
f = \sqrt{1 - f_1^2 - f_2^2}, \quad g = f_2\cos\theta - f_1\sin\theta, \quad h = \sqrt{1 - g^2},
\]
then the immersion 
\[
\phi: S^k \times S^l \times \mathbb{R} \longrightarrow S^{n+1}, \quad \phi(y, z, t) = (f(t)\, y, f_2(t)z, f_1(t))
\]
defines a hypersurface with constant mean curvature \( H \).
\end{theorem}

The following lemma will be key in finding embedded and properly immersed examples.

\begin{lemma}\label{theLemma}
If a solution of the ODE in Theorem \ref{theODE} satisfies the conditions \( \theta(0) = 0 \), \( f_2(0) = a > 0 \), and \( \theta(T) = \pi \), with \( 0 < f_2(t) < 1 \) and \( f_1(t)^2 + f_2^2(t) < 1 \) for \( 0 < t < T \), then \( f_1 \) and \( f_2 \) are \( 2T \)-periodic and satisfy the symmetry relations:
\[
f_1(T + t) = -f_1(T - t), \quad f_2(T - t) = f_2(T + t), \quad \text{for } 0 < t < T.
\]
\end{lemma}

\begin{proof}
A direct verification shows that the functions 
\[
\tilde{f}_1(t) = -f_1(T - t), \quad \tilde{f}_2(t) = f_2(T - t), \quad \tilde{\theta}(t) = 2\pi - \theta(T - t)
\]
satisfy the differential equation \eqref{thesystem} with initial conditions 
\[
\tilde{f}_1(0) = f_1(T) = 0, \quad \tilde{f}_2(0) = f_2(T), \quad \tilde{\theta}(0) = \theta(T) = \pi.
\]
Since the system is autonomous, it follows that 
\[
\tilde{f}_1(t) = -f_1(T - t) = f_1(T + t), \quad \tilde{f}_2(t) = f_2(T - t) = f_2(T + t).
\]
Thus, the functions \( f_1 \) and \( f_2 \) satisfy the stated symmetry properties, completing the proof.
\end{proof}


\subsection{Description of the new CMC hypersurfaces in the sphere.}

This paper shows numerical evidence of the following theorem. 

\begin{theorem} For any pair of positive integers $n$ and $l$ with $l<n-2$, there exists numbers

$$0<a_{n\ell}^{H_{min}}<a_{n\ell}^{H=0}<a_{n\ell}^*<1\quad\hbox{and} \quad H_{n\ell}^{\min}<0$$

and smooth functions 

$$H_{in}:(0,a_{n\ell}^*)\rightarrow [H_{n\ell}^{\min},\infty)\quad\hbox{and} \quad T_{in}:(0,a_{n\ell}^*)\rightarrow (0,\infty)\, ,$$

such that  for any $a\in (0,a_{n\ell}^*)$, the solution of the ODE system \eqref{thesystem} with $H=H_{n\ell}(a)$ that satisfy

$$f_1(0)=0,\quad f_2(0)=a\quad\hbox{and}\quad \theta(0)=0$$

defines a profile curve 

$$\alpha(t)=(f_1(t),f_2(t))$$ 

that is 
a simple closed curve with length $2 \, T_{n\ell}(a)$ and is symmetry with respect to the $f_2$-axis. See Figures \ref{Hnegp1}, \ref{Hnegp2} and \ref{Hpos}. 

Therefore the immersion \ref{param} defines an embedded hypersurface with constant mean curvature $H=H_{n\ell}(a)$ from $S^{n-l-2}\times S^l\times S^1$ into $S^{n+1}$.

Moreover we have,

\begin{enumerate}

\item

$$f_2(T_{n\ell}(a))>a_{n\ell}^*,\quad f_1(T_{n\ell}(a))=0\quad \hbox{and}\quad \theta(T_{n\ell}(a))=\pi$$

\item

 The profile curve when $a$ is near $a_{n\ell}^*$ is a small (with length $2T_{n\ell}(a)$) closed simple curve that goes around the point $(0,a_{n\ell}^*)$. See Figure \ref{Hpos}. We have that 
 
$$\lim_{a\to a_{n\ell}^*} T_{n\ell}(a)=0\quad \hbox{and}\quad \lim_{a\to a_{n\ell}^*} H_{n\ell}(a)=\infty$$
 
 \item
 
 The function $H_{n\ell}(a)$ is decreasing in the interval $(0,a_{n\ell}^{H_{\min}})$ and increasing in the interval $(a_{n\ell}^{H_{\min}},a_{n\ell}^*)$. We have that
 
 $$H_{n\ell}(a_{n\ell}^{H_{\min}})=H_{n\ell}^{\min}<0$$
 
 \item
 
We have that 

$$H_{n\ell}(a_{n\ell}^{H=0})=0$$

Therefore, the is a unique minimal embedded example in this particular family. See Figure \ref{pcHeq0neq3}.
 \item
 
 The profile curve $\alpha(t)$ converges to the vertical segment $\{(0,y):0\le y\le 1\}$ when $a$ approaches zero, see Figure \ref{Hnegp1}. We have that

 $$\lim_{a\to 0^+}H_{n\ell}(a)=0 \quad\hbox{and}\quad \lim_{a\to 0^+}T_{n\ell}(a)=1$$
 
\end{enumerate}

\end{theorem}

\begin{remark} Table \ref{tab:asAndMinH} provide estimates for the values of $0<a_{n\ell}^{H_{min}}<a_{n\ell}^{H=0}<a_{n\ell}^*<1\quad\hbox{and} \quad H_{n\ell}^{\min}<0$ for some values of $(n,l)$. Recall that $n$ is the dimension of the manifold.
\end{remark}

\subsection{Numerical construction of the embedded CMC hypersurfaces }

Motivated by Lemma \ref{theLemma}, for fixed values of \( k \) and \( l \), we consider solutions of the ODE system \eqref{thesystem} that satisfy the initial conditions
\begin{eqnarray}\label{ic}
f_1(0) = 0, \quad f_2(0) = a, \quad \theta(0) = 0.
\end{eqnarray}

We define the functions
\begin{eqnarray}
F_1(a,H,t) = f_1(t), \quad F_2(a,H,t) = f_2(t), \quad \Theta(a,H,t) = \theta(t),
\end{eqnarray}
where \( f_1 \), \( f_2 \), and \( \theta \) satisfy the ODE system \eqref{thesystem} with initial conditions \eqref{ic}. 

Notice that the functions \( F_1 \), \( F_2 \), and \( \Theta \) are smooth and well-defined on an open subset of Euclidean space \( \{(a,H,t) \in \mathbb{R}^3\} \). According to Lemma \ref{theLemma}, we can find properly immersed CMC hypersurfaces by  solving the following system of two equations with three variables:
\begin{eqnarray}\label{system1}
\begin{cases} 
F_1(a,H,t) &= 0, \\ 
\Theta(a,H,t) &= \pi. 
\end{cases}
\end{eqnarray}

We can mathematically establish the existence of solutions to the system above by using a numerical method that tracks the round-off error. In \cite{Pna1} and \cite{Pna2}, the author provides examples of how to rigorously derive mathematical solutions from numerical approximations. This approach, however, will be addressed in a separate paper.

As expected, solving a system of two equations with three variables (\( a, H, t \)) yields a curve. To find this curve, we compute the gradients of the functions \( F_1 \) and \( \Theta \) and apply the techniques developed in \cite{Pds} to graph a solution curve \( \Gamma \) in the space \( (a, H, T) \) that satisfies system \eqref{system1}. Each point \( (a, H, T) \in \Gamma \) corresponds to an embedded hypersurface with constant mean curvature \( H \) and a profile curve of length \( 2T \). 

Figure \ref{Gammaneq3} (left) shows a portion of this curve \( \Gamma \) for the case \( n = 3 \) with \( l = k = 1 \).

For completeness, we briefly outline the procedure. Suppose we have a numerical solution \( (a_0, H_0, T_0) \) to system \eqref{system1}. If we denote \( K_{f_1} \), \( K_{f_2} \), and \( K_{\theta} \) as the partial derivatives of the function \( K \) given in Equation \eqref{dth}, then the following system of differential equations

\begin{eqnarray}\label{thesystemforpartialwrta}
\begin{cases} 
f_1^\prime &= \cos\theta,\\
f_2^\prime &= \sin\theta,\\
\theta^\prime &= K,\\
f_{1a}^\prime &= -\theta_a \sin\theta,\\
f_{2a}^\prime &= \theta_a \cos\theta,\\
\theta_a^\prime &= K_{f_1} f_{1a} + K_{f_2} f_{2a} + K_{\theta} \theta_a
\end{cases}
\end{eqnarray}

with initial conditions 
\[
f_1(0) = 0, \quad f_2(0) = a_0, \quad \theta(0) = 0, \quad f_{1a}(0) = 0, \quad f_{2a}(0) = 1, \quad \theta_a(0) = 0
\]
allows us to compute the partial derivatives of the functions \( F_1 \), \( F_2 \), and \( \Theta \) with respect to \( a \).

Similarly, the following system

\begin{eqnarray}\label{thesystemforpartialwrtH}
\begin{cases} 
f_1^\prime &= \cos\theta,\\
f_2^\prime &= \sin\theta,\\
\theta^\prime &= K,\\
f_{1H}^\prime &= -\theta_H \sin\theta,\\
f_{2H}^\prime &= \theta_H \cos\theta,\\
\theta_H^\prime &= K_{f_1} f_{1H} + K_{f_2} f_{2H} + K_{\theta} \theta_H
\end{cases}
\end{eqnarray}

with initial conditions 
\[
f_1(0) = 0, \quad f_2(0) = a_0, \quad \theta(0) = 0, \quad f_{1H}(0) = 0, \quad f_{2H}(0) = 0, \quad \theta_H(0) = 0
\]
allows us to compute the partial derivatives of \( F_1 \), \( F_2 \), and \( \Theta \) with respect to \( H \).

Once these partial derivatives are obtained, we compute the gradients of the functions \( F_1 \) and \( \Theta \) at a particular solution \( p_1 = (a_1, H_1, T_1) \) of system \eqref{system1}. We then define the vector
\[
v = \nabla F_1(p_1) \times \nabla \Theta(p_1).
\]
If \( v \neq 0 \), then by the implicit function theorem, the solution set \( \Gamma \) of system \eqref{system1} is an embedded curve, with a tangent line at \( p_1 \) parallel to \( v \). Using this observation, the numerical method searches for a solution with the desired precision near the point \( p_1 + h v \) for small \( h \).

\subsubsection{CMC Hypersurfaces in \( S^4 \)}

When \( n = 3 \), the only possible values for \( k \) and \( l \) are \( k = l = 1 \). We can verify that there exists a solution of the form \( q_0 = (0.187605..., 0, 1.15925...) \) that satisfies system \eqref{system1}. Figure \ref{pcHeq0neq3} shows the graphs of the functions \( f_1 \), \( f_2 \), and \( \theta \). 

Notice that the first entry of \( q_0 \) provides the initial condition for \( f_2 \), while the last entry gives half of the period of the functions \( f_1 \) and \( f_2 \). Additionally, \( q_0 \) satisfies system \eqref{system1} because \( f_1(1.15925...) = 0 \) and \( \theta(1.15925...) = \pi \).  

\begin{figure}[h]
\centerline{\includegraphics[scale=0.52]{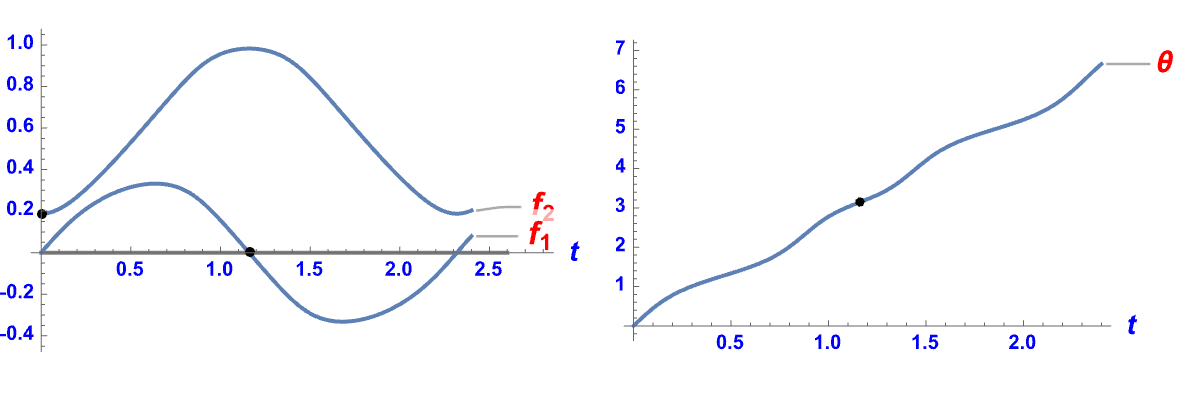} }
\caption{Graphs of the functions \( f_1 \), \( f_2 \), and \( \theta \) associated with the point \( q_0 \). Notice that there are two highlighted points in the left graph and one in the right graph. These indicate that \( f_2(0) = 0.187605... \), which is the first entry of \( q_0 \), and that \( f_1(1.15925...) = 0 \) and \( \theta(1.15925...) = \pi \).}
\label{pcHeq0neq3}
\end{figure}

Recall that the immersion is given by
\[
\phi(u,v,t) = (f(t) \cos u, f(t) \sin u, f_2(t) \cos v, f_2(t) \sin v, f_1(t)).
\]

Since the second entry of \( q_0 \) is zero, the immersion \( \phi \) defines an embedded minimal hypersurface in \( S^4 \). Figure \ref{pcHeq0neq3} shows the profile curve for the immersion. We note that since \( t \) is the arc-length parameter, the length of this profile curve is approximately
\[
2 \times 1.15925... \approx 2.31851.
\]

\begin{figure}[h]
\centerline{\includegraphics[scale=0.52]{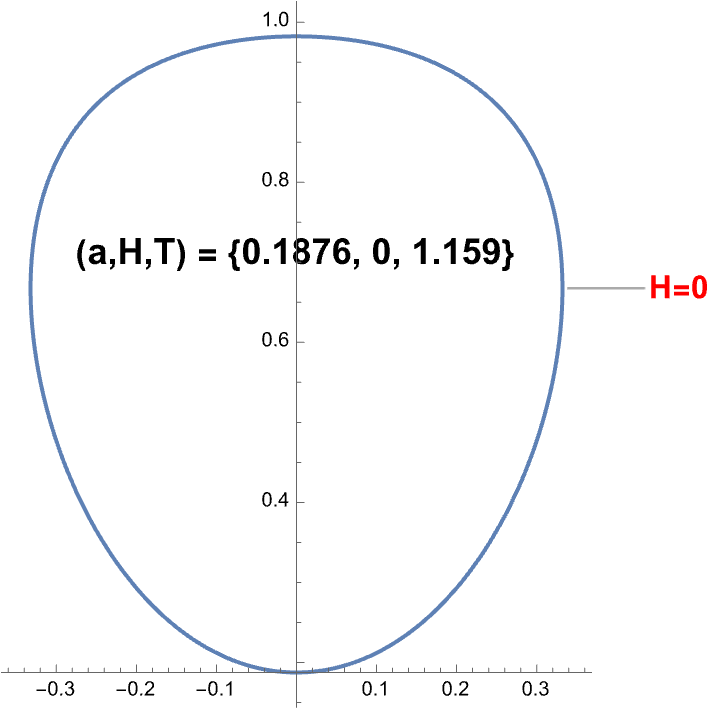} }
\caption{Profile curve of a minimal hypersurface in \( S^4 \).}
\label{pcHeq0neq3}
\end{figure}

This solution was shown to exist mathematically by Carlotto and Schulz in \cite{CS}. They referred to this example as the \emph{hypertorus} and conjectured that it is the \emph{only} embedded minimal example in \( S^4 \) that is topologically equivalent to \( S^1 \times S^1 \times S^1 \).

Using the ODE systems \eqref{thesystemforpartialwrta} and \eqref{thesystemforpartialwrtH}, we compute

\[
\nabla F_1(q_0) \approx (0.966592, -0.883772, -1), \quad \nabla \Theta(q_0) \approx (-0.287382, 0.505903, 1.92866),
\]
and
\[
v = \nabla F_1(q_0) \times \nabla \Theta(q_0) \approx (-1.19859, -1.57684, 0.235021).
\]

Notice that the numerical computations above allow us to prove the following result, which is similar in spirit—though different in technique—to the 1990 result of Brito and Leite \cite{BL}, where they established the existence of rotational embedded CMC hypersurfaces for small values of \( H \).

\begin{theorem}
There exists a positive \( \epsilon \) such that for any \( H \) with \( |H| < \epsilon \), there exists an embedded immersion of \( S^1 \times S^1 \times S^1 \) into \( S^4 \) with constant mean curvature \( H \).
\end{theorem}

\begin{proof} 
From the work of Carlotto and Schulz \cite{CS}, we obtain bounds on the values of \( a \) and \( T \) that satisfy system \eqref{system1} for \( H = 0 \). We can show that if \( a \) and \( T \) remain within these bounds, then \( \nabla F_1 \times \nabla \Theta \) is not the zero vector. Therefore, by the implicit function theorem, system \eqref{system1} admits a solution for values of \( (a, H, T) \) near \( q_0 \).
\end{proof}

The theorem above can be extended whenever \( k = l \), since Carlotto and Schulz have mathematically established the existence of embedded examples with \( H = 0 \).

\subsubsection{Describing the profile curve for the generalize rotational  embedded CMC in $S^4$}

We can use analytic continuation, as in \cite{Pds}, to extend the point \( q_0 \) into a curve \( \Gamma \subset \{(a,H,T) : a,H,T \in \mathbb{R} \} \) such that every point in \( \Gamma \) satisfies system \eqref{system1}. Figure \ref{Gammaneq3} illustrates this curve \( \Gamma \). Every point \( (a,H,T) \in \Gamma \) represents an embedded example.

We observe that for \textit{positive values of} \( H \), as \( H \) increases, the length of the profile curve decreases and eventually converges to zero. Additionally, the negative values of \( H \) are bounded below by a critical value \( H_{31}^{\text{min}} < 0 \). Interestingly, every  \textit{negative value of} \( H \), except \( H_{31}^{\text{min}} \), corresponds to two distinct initial conditions \( a \), leading to two different embedded examples with the same constant mean curvature.

Let us describe further this family of CMC hypersurfaces. Starting at the hypertorus which correspond to the point $q_0$ in $\Gamma$ we have two directions along $\Gamma$. For one of this directions, the values of $a$ star to decrease as well as the values of $H$. $H$ decreases until it reaches a value $H_{31}^{\text{min}}= -0.0798\dots$, for this minimum value of $H$, the value of $a$ has decreased from $a=a_{31}^{H=0}=0.1876\dots$  to $a=a_{31}^{H_{31}^{\min}}=0.0747\dots$. See Figure \ref{Hnegp1}.

Continuing along the same direction of the curve $\Gamma$, the values of $a$ continues to approach zero while the mean curvature $H$ remains negative and also approaches zero. Meanwhile, as $H$ approaches zero, the profile curve approaches a curve close to two vertical lines near the $f_2$-axis. See Figure \ref{Hnegp2}.

When we move in the other direction of $\Gamma$, the value of $a$ increases as well as the values for $H$. Numerical evidence indicates that there exists a value $a_{3,1}^{*}$ such that for every positive $H$ there exists a point in $\Gamma$ with $a$ between $a_{31}^{H=0}$ and $a_{31}^*$ that produces an embedded example with CMC $H$,  see Figure \ref{Hpos}. Recall that the case $k=l$ and $H>0$ was considered by Huang and Wei in \cite{HW}.

\begin{figure}[H]
\centerline{
\includegraphics[scale=0.62]{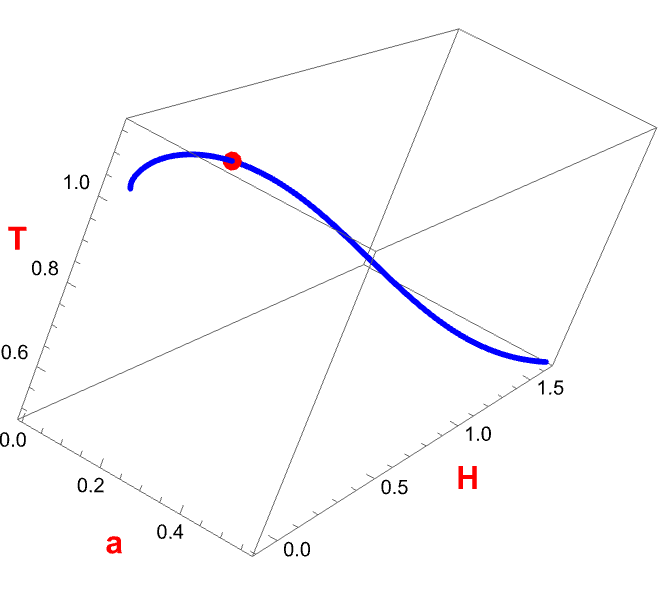} \hskip0.2cm 
\includegraphics[scale=0.62]{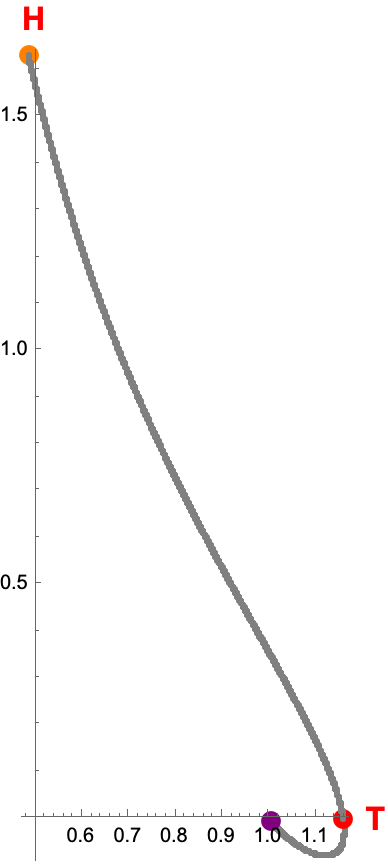}
}
\caption{On the left, we depict the curve \( \Gamma \), consisting of points that solve system \eqref{system1} for \( n=3 \). On the right, we show the projection of \( \Gamma \) onto the \( T \)-\( H \) plane.}
\label{Gammaneq3}
\end{figure}

\section{The Volume some of the minimal examples}\label{area} 

In this section we will deduce a formula for the volume of the hypersurfaces that we are considering. Recall the these hypersurfaces are of the form

\[
\phi: S^k \times S^l \times \mathbb{R} \longrightarrow S^{n+1}, \quad \phi(y, z, t) = (f(t)\, y, f_2(t)z, f_1(t))
\]

with $f=\sqrt{1-f_1^2-f_2^2}$ and $(f_1(t),f_2(t))$ parametrize by arc length. We have that 

$$|\phi_t|=\sqrt{1+\left(\frac{df}{dt}\right)^2}$$

and the element of volume $\text{dVol}$ is given by

$$\text{dVol}=f_2^l*\text{ dVol}_{S^l}*f^k* \text{dVol}_{S^k}*\sqrt{1+\left(\frac{df}{dt}\right)^2}$$

Since the profile curve has period $2T$ and is symmetric with respect the the vertical axis, we obtain that 

$$\text{Volume} = 2 \sigma_k \sigma_l \, \int_0^T f_2^l\, f^k \, \sqrt{1+\left(\frac{df}{dt}\right)^2}\, dt$$

Since 

$$\frac{df}{dt}=-\frac{1}{f}\left(f_1\cos\theta+f_2\sin\theta\right)\, ,$$

then the formula for the volume simplifies to

$$\text{Volume} = 2 \sigma_k \sigma_l \, \int_0^T f_2^l\, f^{k-1} \, \sqrt{1-f_2^2\cos^2\theta-f_2^2\cos^2\theta+2 f_1f_1\cos\theta\sin\theta}\, dt$$

\begin{definition} Let as denote by $\text{Vol }(n,l)$ the $n$-volume of the minimal hypersurface in the one-parametric family studied in this paper  topologically equivalent to $S^{n-l-1}\times S^l\times S^1$. Let us denote by $\text{VolC }(n,l)$ the $n$-volume of the minimal Clifford hypersurface topologically equivalent to $S^{n-l}\times S^l$. Notice that 

$$\text{VolC }(n,l)=\left(\sqrt{\frac{n-l}{n}}\right)^{n-l}*\sigma_{n-1}*\left(\sqrt{\frac{l}{n}}\right)^{l}*\sigma_l$$

\end{definition}

\subsection{Case $n=3$} For $n=3$, this is for hypersurfaces in $S^4$, we have that

 $$\text{VolC }(3,1)=\frac{16 \pi ^2}{3 \sqrt{3}}\approx 30.3905$$
while

 $$\text{Vol }(3,1)\approx 37.8540$$ 

\subsection{Case $n=4$} For $n=4$, this is for hypersurfaces in $S^5$, we have that

$$\text{VolC }(4,1)=\frac{3 \sqrt{3} \pi ^3}{4}\approx 40.2783\quad\hbox{and}\quad \text{VolC }(4,2)= 4 \pi^2\approx 39.4784$$

while  

$$\text{Vol }(4,1)\approx 49.4826$$

\subsection{Case $n=5$} For $n=5$, this is for hypersurfaces in $S^6$, we have that

$$\text{VolC }(5,1)=\frac{256 \pi ^3}{75 \sqrt{5}}\approx 47.3307\quad\hbox{and}\quad \text{VolC }(5,2)= \frac{48}{25} \sqrt{\frac{3}{5}} \pi ^3 \approx 46.1133$$

while  

$$\text{Vol }(5,1)\approx 57.8986\quad\hbox{and}\quad \text{Vol }(5,2)\approx 56.9862$$ 

\subsection{Case $n=6$} For $n=6$, this is for hypersurfaces in $S^7$, we have that

$$
\text{VolC }(6,1)=\frac{2\pi^4\,5^{5/2}}{216}\approx 50.4198\quad\hbox{and}\quad \text{VolC }(6,2)=\frac{128 \pi ^3}{81} \approx 48.9976,
$$

while

$$
\text{Vol }(6,1)\approx 61.5653\quad\hbox{and}\quad \text{Vol }(6,2)\approx 60.2932.
$$

\subsection{Case $n=7$} For $n=7$, this is for hypersurfaces in $S^8$, we have that

$$
\begin{aligned}
\text{VolC }(7,1)&=\frac{2304\,\pi^4}{1715\sqrt{7}}\approx 49.4617,\\[1mm]
\text{VolC }(7,2)&=\frac{200}{343}\sqrt{\frac{5}{7}}\,\pi^4\approx 48.0033,\\[1mm]
\text{VolC }(7,3)&=\frac{256}{343}\sqrt{\frac{3}{7}}\,\pi^4\approx 47.5945,
\end{aligned}
$$

while

$$
\begin{aligned}
\text{Vol }(7,1)&\approx 60.3392,\\[1mm]
\text{Vol }(7,2)&\approx 58.9648,\\[1mm]
\text{Vol }(7,3)&\approx 58.6727.
\end{aligned}
$$

\subsection{Case $n=8$} For $n=8$, this is for hypersurfaces in $S^9$, we have that

\[
\begin{aligned}
\text{VolC }(8,1)&=\frac{343\sqrt{7}\,\pi^5}{6144}\approx 45.2003,\\[1mm]
\text{VolC }(8,2)&=\frac{9\,\pi^4}{20}\approx 43.8341,\\[1mm]
\text{VolC }(8,3)&=\frac{75\sqrt{15}\,\pi^5}{2048}\approx 43.4037,\\[1mm]
\text{VolC }(8,4)&=\frac{4\,\pi^4}{9}\approx 43.2929.
\end{aligned}
\]

while

\[
\begin{aligned}
\text{Vol }(8,1)&\approx 55.1111,\\[1mm]
\text{Vol }(8,2)&\approx 53.7953,\\[1mm]
\text{Vol }(8,3)&\approx 53.4141.
\end{aligned}
\]

\subsection{Case $n=9$} For $n=9$, this is for hypersurfaces in $S^{10}$, we have that

$$
\begin{aligned}
\text{VolC }(9,1)&=\frac{262144\,\pi^5}{2066715}\approx 38.8158,\\[1mm]
\text{VolC }(9,2)&=\frac{2744\sqrt{7}\,\pi^5}{59049}\approx 37.6244,\\[1mm]
\text{VolC }(9,3)&=\frac{256\,\pi^5}{1215\sqrt{3}}\approx 37.2265,\\[1mm]
\text{VolC }(9,4)&=\frac{3200\sqrt{5}\,\pi^5}{59049}\approx 37.0827.
\end{aligned}
$$

while

$$
\begin{aligned}
\text{Vol }(9,1)&\approx 47.3107,\\[1mm]
\text{Vol }(9,2)&\approx 46.1509,\\[1mm]
\text{Vol }(9,3)&\approx 45.7730,\\[1mm]
\text{Vol }(9,4)&\approx 45.6751.
\end{aligned}
$$

\subsection{Case $n=10$} For $n=10$, this is for hypersurfaces in $S^{11}$, we have that

$$
\begin{aligned}
\text{VolC }(10,1)&=\frac{6561\,\pi^6}{200000}\approx 31.5384,\\[1mm]
\text{VolC }(10,2)&=\frac{32768\,\pi^5}{328125}\approx 30.5605,\\[1mm]
\text{VolC }(10,3)&=\frac{343\sqrt{21}\,\pi^6}{50000}\approx 30.2227,\\[1mm]
\text{VolC }(10,4)&=\frac{1536\,\pi^5}{15625}\approx 30.0830,\\[1mm]
\text{VolC }(10,5)&=\frac{\pi^6}{32}\approx 30.0434.
\end{aligned}
$$

while

$$
\begin{aligned}
\text{Vol }(10,1)&\approx 38.4320,\\[1mm]
\text{Vol }(10,2)&\approx 37.4742,\\[1mm]
\text{Vol }(10,3)&\approx 37.1433,\\[1mm]
\text{Vol }(10,4)&\approx 37.0227.
\end{aligned}
$$

\subsection{Case $n=11$} For $n=11$, this is for hypersurfaces in $S^{12}$, we have that

$$
\begin{aligned}
\text{VolC }(11,1)&=\frac{2560000\,\pi^6}{30438639\sqrt{11}}\approx 24.3791,\\[1mm]
\text{VolC }(11,2)&=\frac{13122\,\pi^6}{161051\sqrt{11}}\approx 23.6178,\\[1mm]
\text{VolC }(11,3)&=\frac{262144\sqrt{\frac{3}{11}}\,\pi^6}{5636785}\approx 23.3492,\\[1mm]
\text{VolC }(11,4)&=\frac{43904\sqrt{\frac{7}{11}}\,\pi^6}{1449459}\approx 23.23,\\[1mm]
\text{VolC }(11,5)&=\frac{5760\sqrt{\frac{5}{11}}\,\pi^6}{161051}\approx 23.1818.
\end{aligned}
$$

while

$$
\begin{aligned}
\text{Vol }(11,1)&\approx 29.7032,\\[1mm]
\text{Vol }(11,2)&\approx 28.9549,\\[1mm]
\text{Vol }(11,3)&\approx 28.6874,\\[1mm]
\text{Vol }(11,4)&\approx 28.5757,\\[1mm]
\text{Vol }(11,5)&\approx 28.5441.
\end{aligned}
$$

\subsection{Case $n=12$} For $n=12$, this is for hypersurfaces in $S^{13}$, we have that

$$
\begin{aligned}
\text{VolC }(12,1)&=\frac{161051\sqrt{11}\,\pi^7}{89579520}\approx 18.0094,\\[1mm]
\text{VolC }(12,2)&=\frac{2500\,\pi^6}{137781}\approx 17.4442,\\[1mm]
\text{VolC }(12,3)&=\frac{27\sqrt{3}\,\pi^7}{8192}\approx 17.2418,\\[1mm]
\text{VolC }(12,4)&=\frac{4096\,\pi^6}{229635}\approx 17.1483,\\[1mm]
\text{VolC }(12,5)&=\frac{8575\sqrt{35}\,\pi^7}{8957952}\approx 17.1044,\\[1mm]
\text{VolC }(12,6)&=\frac{4\,\pi^6}{225}\approx 17.0914.
\end{aligned}
$$

while

$$
\begin{aligned}
\text{Vol }(12,1)&\approx 21.9400,\\[1mm]
\text{Vol }(12,2)&\approx 21.3830,\\[1mm]
\text{Vol }(12,3)&\approx 21.1796,\\[1mm]
\text{Vol }(12,4)&\approx 21.0886,\\[1mm]
\text{Vol }(12,5)&\approx 21.0516.
\end{aligned}
$$

\newpage

\section{Graphs}

The images in this page illustrate shows the profile curve of some embedded examples with  values of $a$ between  $a_{31}^*=0.0074\dots$ 
and  $a^{H=0}_{31}=0.1876\dots$.

\begin{figure}[H]
    \centering
    \includepdf[pages=-, scale=0.8, pagecommand={}, offset=0 -30]{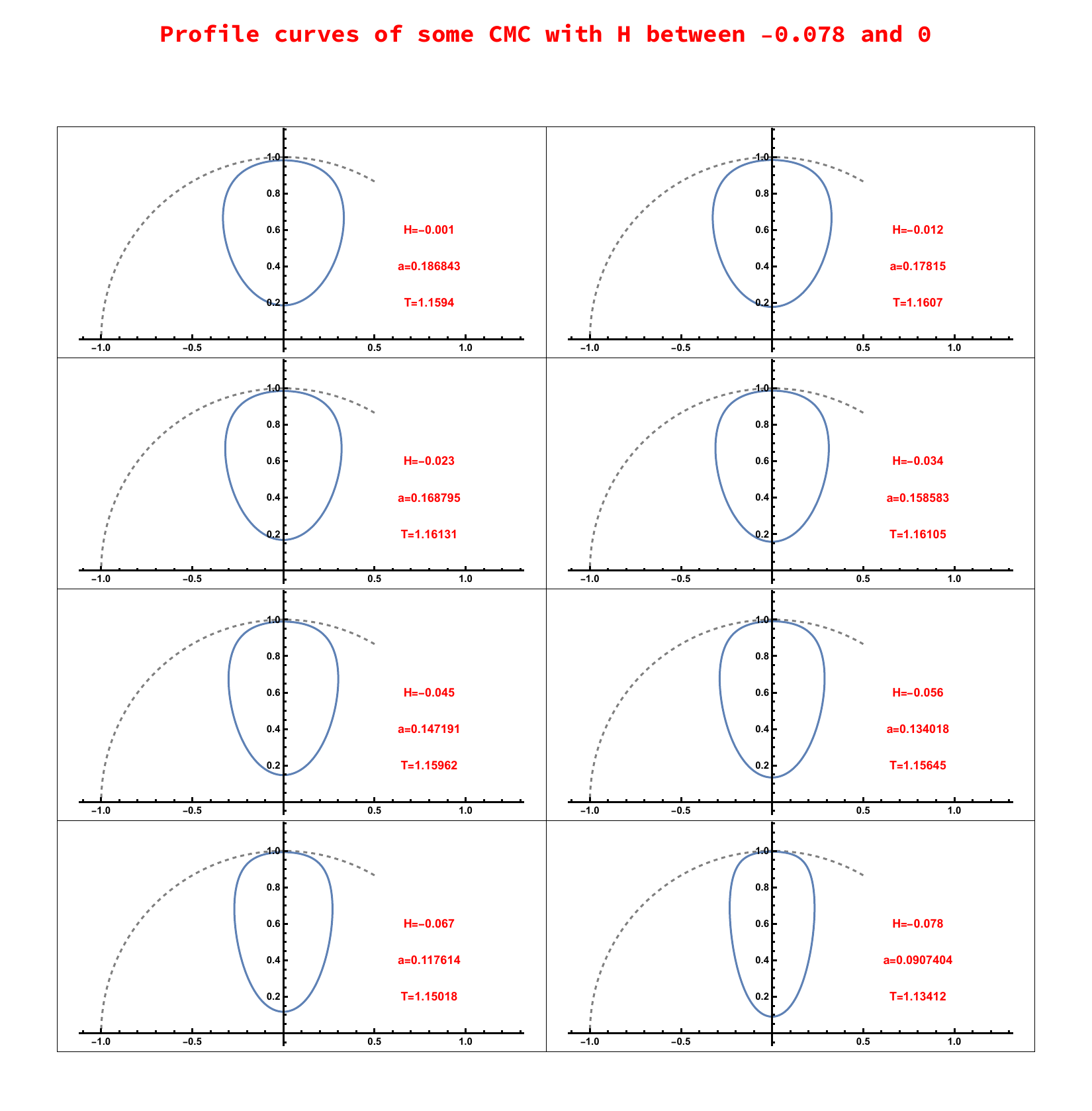}
    \caption{Profile curves with $a\in (a_{31}^{H_{31}^{\min}},a^{H=0}_{31})$.}\label{Hnegp1}
   
\end{figure}
\newpage

The images in this page illustrate shows the profile curve of some embedded examples with  values of $a$ between  to $0$ and $a_{31}^*=0.0074\dots$  
 
\begin{figure}[H]
    \centering
    \includepdf[pages=-, scale=0.8, pagecommand={}, offset=0 -30]{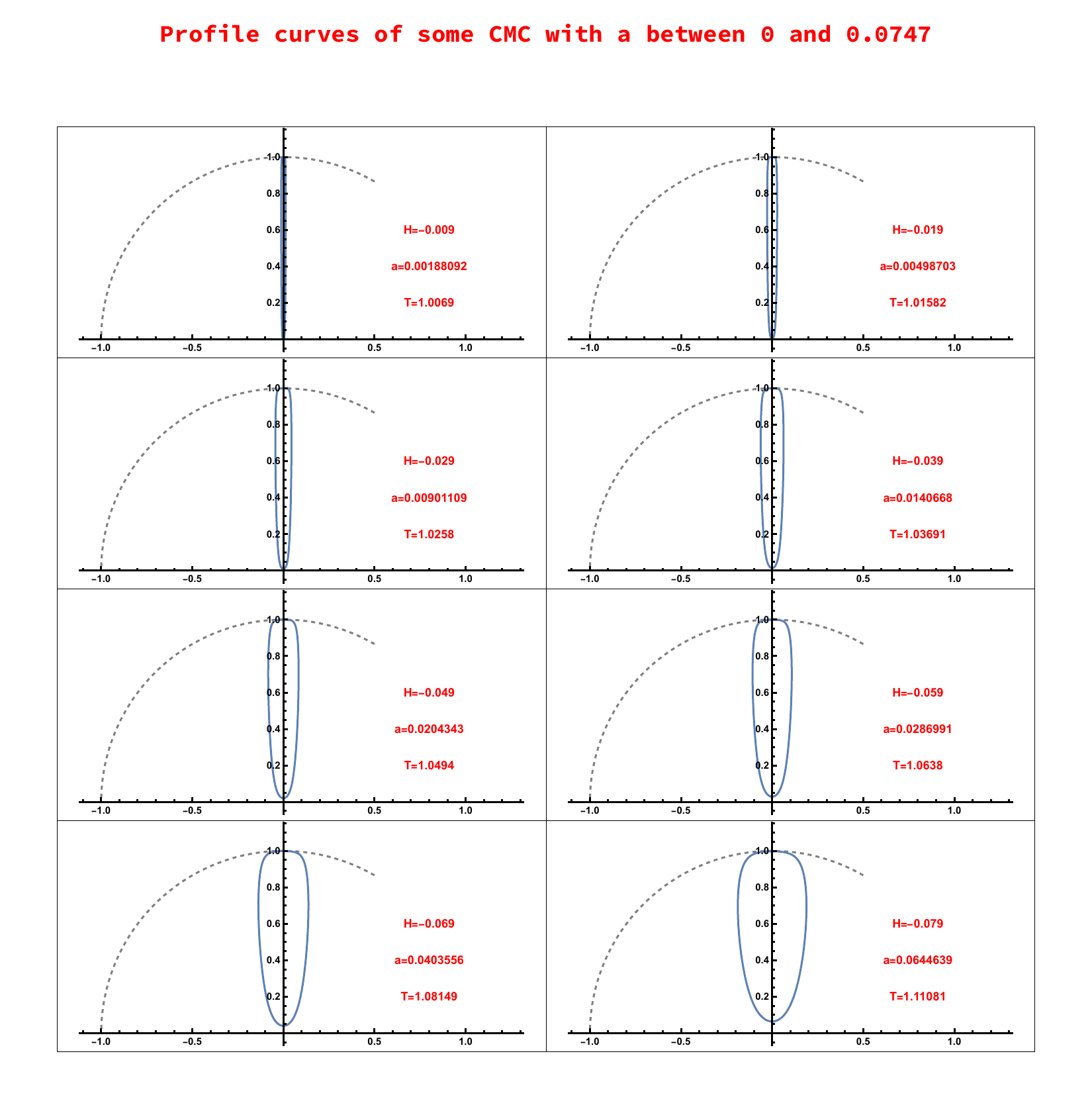}
    \caption{Profile curves with $a\in (0,a_{31}^*)$}\label{Hnegp2}
\end{figure}

\newpage

\begin{figure}[H]
    \centering
    \includepdf[pages=-, scale=0.8, pagecommand={}, offset=0 -30]{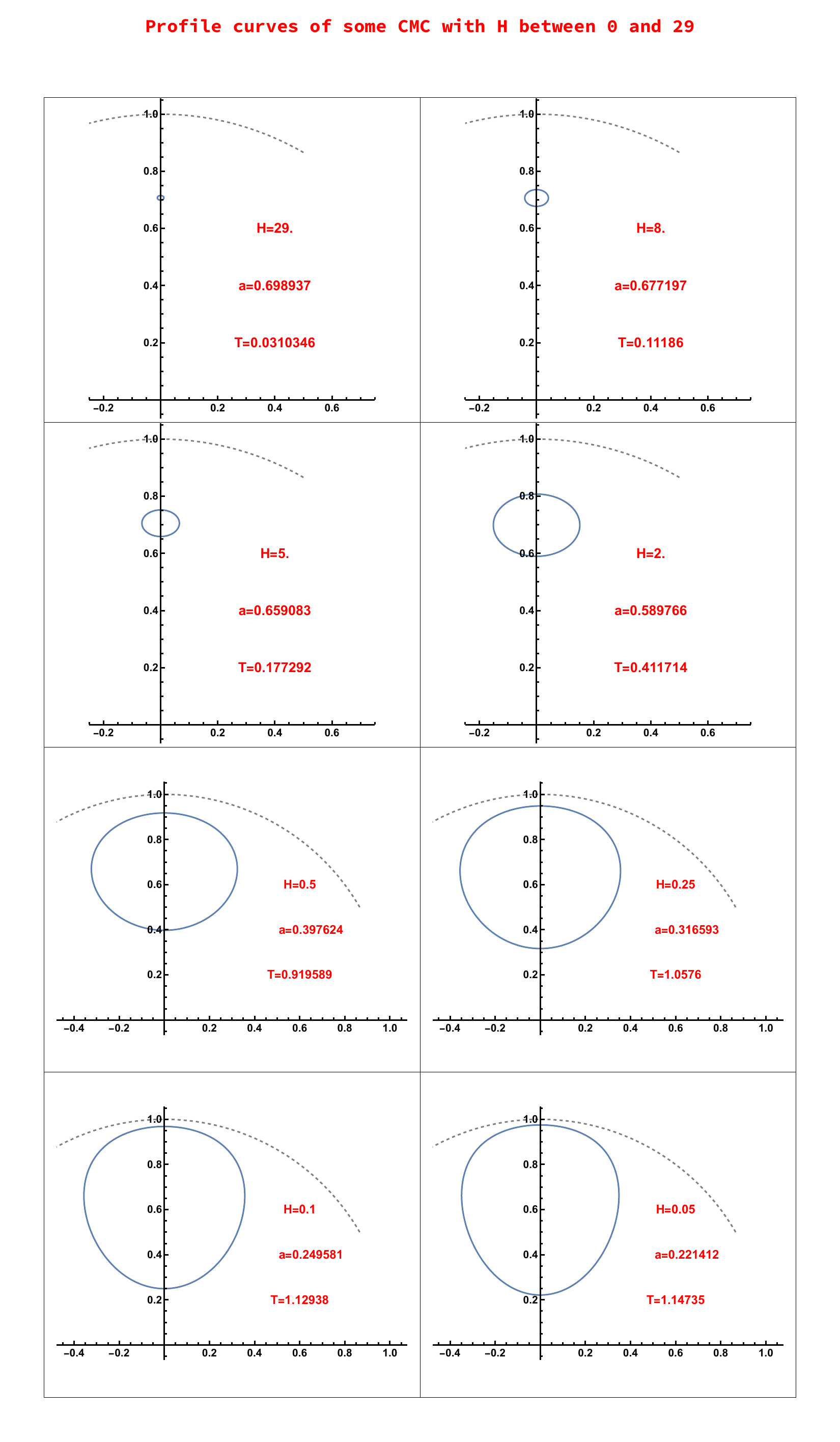}
    \caption{Profile curves with $a\in ( a_{31}^{H=0}, a_{31}^*)$}
  \label{Hpos}
\end{figure}

\newpage

{\smaller
\begin{table}[H]
\renewcommand{\arraystretch}{1.5} 
\centering
\begin{tabular}{|c|c|c|c|c|}
\hline
$(n,\ell)$ & $a_{n\ell}^{H_{\min}}$ & $a_{n\ell}^{H=0}$ & $ a_{n\ell}^* \in(A,B)$ & $ H^{\min}_{n\ell}$ \\
\hline
 \text{(3,1)} & 0.07488 & 0.1876 & (0.69893,0.71518) & -0.07989 \\
 \hline
 \text{(4,1)} & 0.06705 & 0.16853 & (0.56352,0.59113) & -0.09905 \\
 \hline
 \text{(5,1)} & 0.06029 & 0.14971 & (0.48862,0.51141) & -0.09866 \\
 \hline
 \text{(5,2)} & 0.17229 & 0.33098 & (0.6973,0.71677) & -0.11987 \\
 \hline
 \text{(6,1)} & 0.05507 & 0.13538 & (0.43752,0.45696) & -0.0946 \\
 \hline
 \text{(6,2)} & 0.15519 & 0.29683 & (0.62353,0.64131) & -0.12072 \\
 \hline
 \text{(7,1)} & 0.05095 & 0.12431 & (0.3999,0.41665) & -0.09 \\
 \hline
 \text{(7,2)} & 0.14246 & 0.2708 & (0.76683,0.78221) & -0.11714 \\
 \hline
 \text{(7,3)} & 0.23697 & 0.39953 & (0.69991,0.71422) & -0.12333 \\
 \hline
 \text{(8,1)} & 0.04765 & 0.1155 & (0.37071,0.38527) & -0.08561 \\
 \hline
 \text{(8,2)} & 0.13228 & 0.2504 & (0.52737,0.54167) & -0.11253 \\
 \hline
 \text{(8,3)} & 0.21876 & 0.36791 & (0.42153,0.43128.) & -0.12103 \\
 \hline
 \text{(9,1)} & 0.04489 & 0.10829 & (0.6973,0.71677) & -0.08162 \\
 \hline
 \text{(9,2)} & 0.12417 & 0.23392 & (0.49359,0.50642) & -0.1079 \\
 \hline
 \text{(9,3)} & 0.20423 & 0.34261 & (0.60624,0.61847) & -0.11732 \\
 \hline
 \text{(9,4)} & 0.28359 & 0.44143 & (0.70134,0.71282) & -0.11986 \\
 \hline
 \text{(10,1)} & 0.04255 & 0.10226 & (0.32756,0.33914) & -0.07804 \\
 \hline
 \text{(10,2)} & 0.1173 & 0.22029 & (0.46564,0.47718) & -0.10353 \\
 \hline
 \text{(10,3)} & 0.19238 & 0.32184 & (0.57166,0.58302) & -0.11329 \\
 \hline
 \text{(10,4)} & 0.26605 & 0.41362 & (0.6614,0.67189) & -0.11703 \\
 \hline
 \text{(11,1)} & 0.04049 & 0.09713 & (0.31105,0.32143) & -0.07484 \\
 \hline
 \text{(11,2)} & 0.1116 & 0.20877 & (0.44193,0.45251) & -0.09952 \\
 \hline
 \text{(11,3)} & 0.18225 & 0.30441 & (0.54243,0.55301) & -0.10933 \\
 \hline
 \text{(11,4)} & 0.25137 & 0.39044 & (0.6275,0.63738) & -0.11367 \\
 \hline
 \text{(11,5)} & 0.31936 & 0.47029 & (0.70237,0.7118) & -0.11494 \\
 \hline
 \text{(12,1)} & 0.03879 & 0.09269 & (0.29583,0.30724) & -0.07197 \\
 \hline
 \text{(12,2)} & 0.10659 & 0.19888 & (0.42153,0.43128) & -0.09586 \\
 \hline
 \text{(12,3)} & 0.17372 & 0.28953 & (0.51736,0.5271) & -0.10559 \\
 \hline
 \text{(12,4)} & 0.23893 & 0.37077 & (0.59835,0.60767) & -0.11023 \\
 \hline
 \text{(12,5)} & 0.30276 & 0.44584 & (0.66868,0.67968) & -0.11218 \\
 \hline
\end{tabular}
\vskip.2cm
\caption{Relevant values for $a$ and the lowest possible CMC $H$.}
\label{tab:asAndMinH}
\end{table}
}

%
%

 \end{document}